\documentclass[10pt]{bmc_article}

\usepackage{cite} 
\usepackage{url}  
\usepackage{ifthen}  
\usepackage{multicol}   
\usepackage[utf8]{inputenc} 

\usepackage{amsmath}
\usepackage{amssymb}
\usepackage{amsthm}

\newtheorem{theorem}{Theorem}[section]

\newtheorem{corollary}[theorem]{Corollary}
\newtheorem{proposition}[theorem]{Proposition}

\newcommand{\keywords}[1]{\par\noindent 
{\small{\bf Keywords\/}: #1}}

\newcommand{\subjclass}[1]{\par\noindent 
{\small{\bf Mathematics Subject Classification (2010)\/}: #1}}

\def\includegraphics{}

\newboolean{publ}

\newenvironment{bmcformat}{\baselineskip20pt\sloppy\setboolean{publ}{false}}{\baselineskip20pt\sloppy}

\begin{document}
\begin{bmcformat}


\title{Applications of fixed point theorems in the theory of invariant subspaces}
 

\author{Rafa Esp\'{\i}nola%
         \email{Rafa Esp\'{\i}nola - espinola@us.es}
       and 
         Miguel Lacruz\correspondingauthor
         \email{Miguel Lacruz\correspondingauthor - lacruz@us.es}
      }


\address{Departamento de An\'alisis Matem\'atico, Facultad de Matem\'aticas,  Universidad de Sevilla, Avenida Reina Mercedes, 41012 Seville (SPAIN)
}%

\maketitle


\date{\today}

\begin{abstract}
We survey  several applications of fixed point theorems in the theory of invariant subspaces. The general idea is that a fixed point theorem applied to a suitable map yields the existence of invariant subspaces for an operator on a Banach space.
\end{abstract}


\ifthenelse{\boolean{publ}}{\begin{multicols}{2}}{}


\section{Introduction}
\label{intro}
One of the most recalcitrant unsolved  problems in operator theory is the invariant subspace problem. The question has an easy formulation. Does every operator on an infinite dimensional, separable complex Hilbert space have a non trivial invariant subspace?

Despite  the simplicity of its statement,  this is a very difficult problem and it has generated a very large amount of literature. We refer the reader to the  expository paper of  Yadav \cite{yadav} for a detailed account of results related to the invariant subspace problem.

In this survey we discuss some applications of  fixed point theorems in the theory of invariant subspaces. The general idea is that a fixed point theorem applied to a suitable map yields the existence of invariant subspaces for an operator on a Banach space.

In Section \ref{compact} we consider the  striking theorem of Lomonosov \cite{lomonosov73} about the existence of invariant subspaces for algebras containing compact operators. The proof of this theorem is based on the Schauder fixed point theorem.

In Section \ref{local} we present a recent result of Lomonosov, Radjavi, and Troitsky \cite{LRT} about the existence of invariant subspaces for localizing algebras. The proof of this result is based on the Ky Fan fixed point theorem for multivalued maps. The idea of using fixed point theorems for multivalued maps in the search for invariant subspaces was first introduced by Androulakis \cite{androulakis}.

In Section \ref{burnside} we consider an extension of Burnside's theorem to infinite dimensional Banach spaces.  This result is  originally due to Lomonosov \cite{lomonosov91}. We present a proof of it in a special case that was obtained independently by Scott Brown  \cite{brown} and that  once again is based on the Schauder fixed point  \nolinebreak  theorem.

In Section \ref{krein} we address the existence of invariant subspaces  for operators on Krein space of indefinite product, and we present a result of Albeverio, Makarov, and Motovilov \cite{AMM} whose proof uses the Banach fixed point theorem.

The rest of this section contains some notation, a precise statement of the invariant subspace problem, and  a few historical remarks. 

Let \(E\) be an infinite dimensional, complex Banach space and let \({\mathcal B}(E)\) denote the algebra of all bounded linear operators on \(E.\)  A subspace of \(E\) is by definition a closed linear manifold in \(E.\) 

A subspace \(M \subseteq E\) is said to be {\em invariant} under an operator \(T \in {\mathcal B}(E)\) provided that \(TM \subseteq M,\) and a subspace \(M \subseteq E\) is said to be {\em invariant} under a subalgebra \({\mathcal R} \subseteq {\mathcal B}(H)\) provided that \(M\) is invariant under every \(R \in {\mathcal R}.\) A subalgebra \({\mathcal R} \subseteq {\mathcal B}(H)\) is said to be 
{\em transitive} provided that the only subspaces invariant under \({\mathcal R}\) are the trivial ones, \(M=\{0\}\) and \(M=E.\) This is equivalent to say that the subspace  \(\{Rx: R \in {\mathcal R}\}\) is dense in \(E\) for each \(x \in E \backslash \{ 0\}.\)

The {\em commutant} of a set of  operators  \({\mathcal S} \subseteq {\mathcal B}(E)\) is the subalgebra \({\mathcal S}^\prime\)  of all  operators \(R \in {\mathcal B}(E)\) such that \(SR=RS\) for all \(S \in {\mathcal S}.\) A subspace \(M \subseteq E\) is said to be {\em hyperinvariant} under an operator \(T \in {\mathcal B}(E)\) provided that \(M\) is invariant under \(\{T\}^\prime.\)

The {\em invariant subspace problem} is the question of whether every operator in \({\mathcal B}(E)\) has a non trivial  invariant subspace. This is one of the most important open problems in operator theory. 

The  origin of this question goes back to 1935, when  von Neumann proved the  unpublished   result that any compact operator on a Hilbert space has a non trivial invariant subspace. Aronszajn and Smith \cite{AS} extended this result  in 1954 to general Banach spaces. Bernstein and Robinson \cite{BeRo} used non standard analysis to prove in 1966  that every polynomially compact operator on a Hilbert space has a non trivial invariant subspace. Halmos \cite{halmos} obtained a proof of the same result using classical methods. 

Lomonosov \cite{lomonosov73} proved  in 1973 that any non scalar operator on a Banach space that commutes with a non zero compact operator has a non trivial hyperinvariant subspace. The result of Lomonosov  came into the scene like a lightning bolt in a clear sky, generalizing all the previously known results, and introducing the use of the Schauder fixed point theorem as a new technique to produce invariant subspaces.

Enflo \cite{enflo1} constructed in 1976 the first example of an operator on a Banach space without non trivial invariant subspaces. The example circulated in preprint form and it did not appear published until 1987, when it was recognized as correct work \cite{enflo2}. In the meantime, Beauzamy \cite{beauzamy} simplified the technique,  and further examples were given by Read \cite{read1,read2}.

Very recently, Argyros and Haydon \cite{AH}  constructed an example of an infinite dimensional, separable  Banach space such that every continuous operator is the sum of a compact operator and a scalar operator, so that every operator on it has a non trivial invariant subspace. 

However, after so many decades, the question about the existence of invariant subspaces  for operators on Hilbert space is still an open problem.

\section{Invariant subspaces for algebras containing compact operators}
\label{compact}
We start with a fixed point theorem that is the key for the main result in this section. The use of this result is one of the main ideas in the technique of Lomonosov. We shall denote by   \(\overline{{\rm conv}}(S)\) the closed convex hull of a subset \(S \subseteq E.\)  
\begin{proposition}\cite[Proposition 1]{PS}
\label{key}
Let \(E\) be a Banach space, let \(C \subseteq E\) be a closed convex set, and let \(\Phi: C \to E\) be a continuous mapping such that \(\Phi(C)\) is a relatively compact subset of \(C.\) Then there is a point \(x_0 \in C\) such that \(\Phi(x_0)=x_0.\)
\end{proposition}
\begin{proof} Let \(Q\) denote the closure of \(\Phi(C).\) It follows from a theorem of Mazur that \(\overline{{\rm conv}}(Q)\) is a compact, convex subset of \(E,\) and   since  \(C\) is closed and convex,  we have \(\overline{{\rm conv}}(Q) \subseteq C.\)  Since \(\Phi(C) \subseteq Q,\) we have \(\Phi(\overline{{\rm conv}}(Q)) \subseteq Q \subseteq \overline{{\rm conv}}(Q),\) and now the result follows  from the Schauder fixed point theorem.
\end{proof}
\begin{theorem}\cite[Theorem 2]{PS}
\label{main} Let \({\mathcal R} \subseteq {\mathcal B}(E)\) be a transitive algebra  and let \(K \in {\mathcal B}(E)\) be a non zero compact operator. Then there is a operator  \(R \in {\mathcal R}\) and there is a vector \(x_1 \in E\) such that \(RKx_1=x_1.\)
\end{theorem}
\begin{proof}
We may assume without loss of generality that \(\|K\|=1.\) Choose an \(x_0 \in E\) such that \(\|Kx_0\| > 1,\) so that \(\|x_0\| >1.\) Consider the closed ball \(B=\{x \in E:\|x-x_0\| \leq 1\}.\) Then, for each \(R \in {\mathcal R},\) consider the open set \(G_R= \{ y \in E : \|Ry-x_0\| < 1\}.\) Since \({\mathcal R}\) is a transitive algebra, we have
\[
\bigcup_{R \in {\mathcal R}} G_R=E \backslash \{0\}.
\]
Since \(K\) is a compact operator, \(\overline{KB}\) is a compact subset of \(E,\) and since \(\|K\|=1\) and \(\|Kx_0\|>1,\) we have \(0 \notin \overline{KB}.\) Thus, the family  \(\{G_R:R \in {\mathcal R}\}\) is an open cover of  \(\overline{KB}.\) Hence, there exist finitely many operators \(R_1, \ldots , R_n \in {\mathcal R}\) such that
\[
\overline{KB} \subseteq \bigcup_{i=1}^n G_{R_i}.
\]
Next, for each \(y \in \overline{KB}\) and \(i=1, \ldots , n\) we define \(\alpha_i(y)=\max \{0, 1- \|R_iy-x_0\|\}.\) Then \(0 \leq \alpha_i(y) \leq 1,\) and for each \(y \in \overline{KB}\) there is an \(i=1, \ldots , n\) such that \(y \in G_{R_i},\) so that \(\alpha_i(y)>0.\) Thus \(\displaystyle{\sum_{i=1}^n \alpha_i(y)>0}\) for each \(y \in \overline{KB},\) and we may define
\[
\beta_i(y)= \frac{\alpha_i(y)}{\displaystyle{\sum_{j=1}^n\alpha_j(y)}}
\]
for \(i=1, \ldots , n\) and \(y \in \overline{KB}.\) Now, each \(\beta_i\) is a continuous function from \(\overline{KB}\) into \({\mathbb R}.\) Hence, we may define a continuous mapping
\(\Phi:B \to E\) by the expression
\[
\Phi(x)= \sum_{i=1}^n \beta_i(Kx)R_iKx.
\]
We claim that \(\Phi(B) \subseteq B.\) Indeed,  for each \(x \in B\) we have \(\displaystyle{\sum_{i=1}^n \beta_i(Kx)=1,}\) so that
\[
\| \Phi(x) - x_0 \|  =  \left \| \sum_{i=1}^n \beta_i(Kx)(R_iKx -x_0)\right \| \leq  \sum_{i=1}^n \beta_i(Kx) \|R_iKx-x_0\|.
\]
If \(\|R_iKx-x_0\| > 1\) then \(\alpha_i(Kx)=0\) and therefore \(\beta_i(Kx)=0.\) Hence
\[
\| \Phi(x) - x_0 \|  \leq  \sum_{i=1}^n \beta_i(Kx) =1,
\]
and this completes the proof of our  claim. Finally, each operator \(R_iK\) is compact, so that each \(R_iKB\) is relatively compact, and it follows from an earlier mentioned theorem of Mazur
that \(\displaystyle{Q=\overline{{\rm conv}}  \bigcup_{i=1}^n R_iKB }\) is compact. Since \(\Phi(B) \subseteq Q,\) the set \(\Phi(B)\) is a relatively compact subset of \(B.\) Now we apply Proposition \ref{key} to find a vector 
\(x_1 \in B\) such that \(\Phi(x_1)=x_1.\) Since \(0 \notin B,\) we have \(x_1 \neq 0.\) Then we consider the operator defined by
\[
Rx= \sum_{i=1}^n \beta_i(Kx_1)R_ix,
\]
and we conclude that \(R \in {\mathcal R}\) and \(RKx_1=x_1,\) as we wanted.
\end{proof}
\begin{corollary}\cite{lomonosov73}, \cite[Theorem 3]{PS} Every non scalar operator that commutes with a non zero compact operator has a non trivial,  hyperinvariant subspace.
\end{corollary}
\begin{proof} Let \(T \in {\mathcal B}(E)\) be a non scalar operator and suppose that \(T\) commutes with a nonzero compact operator \(K.\) We must show that the commutant \(\{T\}^\prime\) is non transitive. Suppose, on the contrary, that  \(\{T\}^\prime\) is  transitive. We can apply Theorem \ref{main} to find an operator \(R \in \{T\}^\prime\) such that \(\lambda=1\) is an eigenvalue of the compact operator \(RK\) with associated finite dimensional eigenspace \(F= \ker (RK-I).\) Since \(T\) commutes with \(RK,\) we observe that \(T\)  maps \(F\) into itself and therefore \(T\) must have an eigenvalue. Since \(T\) is non scalar, the corresponding eigenspace \(M\) cannot be the whole \(E\), and it is invariant under \(\{T\}^\prime.\) The contradiction has arrived. 
\end{proof}

\section{Invariant subspaces for localizing algebras}
\label{local}
In this section we use the following fixed point theorem of Ky Fan \cite{fan}. Recall that if \(\Omega\) is a topological space and \(\Phi: \Omega \to {\mathcal P}(\Omega)\) is a point to set map from \(\Omega\) to the power set of \(\Omega,\) then \(\Phi\) is said to be {\em upper semicontinuous} if for every \(x_0 \in \Omega\) and every open set \(U \subseteq \Omega\) such that \(\Phi (x_0) \subseteq U\) there is a neighborhood \(V\) of \(x_0\) such that \(\Phi (x) \subseteq U\) for every \(x \in V.\) In terms of convergence of nets, this definition is equivalent to say that for every \(x \in \Omega,\)  for every net \((x_\alpha)\)  with \(x_\alpha \to x,\) and for every \(y_\alpha \in \Phi(x_\alpha)\) such that  the net \((y_\alpha)\) converges to some \(y \in \Omega,\) we have \(y \in \Phi (x).\)

\begin{theorem}[Ky Fan fixed point theorem \cite{fan}]  Let \(C\) be a compact convex subset of a locally convex space and let \(\Phi \colon C \to {\mathcal P}(C)\) be an upper semicontinuous mapping such that \(\Phi(x)\) is a non empty, closed convex set for every \(x \in C.\) Then there is an \(x_0 \in C\) such that \(x_0 \in \Phi(x_0).\)

\end{theorem}
A subalgebra \({\mathcal R} \subseteq {\mathcal B}(H)\) is said to be {\em strongly compact} if its unit ball is precompact in the strong operator topology. An important example of a strongly compact algebra is the commutant of a compact operator with dense range. We shall denote by \({\rm ball }(\mathcal{R})\) the unit ball of \(\mathcal{R}.\)

This notion was introduced by Lomonosov \cite{lomonosov80} as a means to prove the existence of  invariant subspaces for essentially normal operators on Hilbert spaces.
Recall that an operator \(T\) on a Hilbert space is said to be {\em essentially normal} if \(T^\ast T-TT^\ast\) is a compact operator.
Lomonosov showed that if an essentially normal operator \(T\) has the property that both its commutant \(\{T\}^\prime\) and the commutant of its adjoint \(\{T^\ast\}^\prime\) fail to be strongly compact, then  \(T\) has a nontrivial invariant subspace.

Thus, in order to solve the invariant subspace problem for essentially normal operators, it suffices to consider only operators with a strongly compact commutant. 

Lomonosov, Radjavi and Troitsky \cite{LRT} obtained a result about the existence of invariant subspaces for an operator with a strongly compact commutant under the additional assumption  that the commutant of the adjoint is a localizing algebra.

A subalgebra \({\mathcal R} \subseteq {\mathcal B}(E)\) is said to be {\em localizing} provided that there is a closed ball \(B \subseteq E\) such that \( 0 \notin B\) and such that for every sequence \((x_n)\) in \(B\) there is a subsequence \((x_{n_j})\) and a sequence of operators \((R_j)\) in \({\mathcal R}\) such that \(\|R_j\| \leq 1\) and \((R_jx_{n_j})\) converges in norm to some non zero vector. An important example of a localizing algebra is any algebra  containing a non zero compact operator.

\begin{proposition}\cite[Proof of Theorem 2.3]{LRT}
\label{translocal}
Let \({\mathcal R} \subseteq {\mathcal B}(E)\) be a transitive localizing algebra, let \(B \subseteq E\) be a closed ball as above,
and let \(T \in {\mathcal R}^\prime\) be a nonzero operator. 
Then there exists an \(r>0\) such that  for every \(x \in B\) we have \(r \,{\rm ball}({\mathcal R})(Tx) \cap B \neq \emptyset.\)
\end{proposition}
\begin{proof}
First, \(T\) is one to one, because \({\mathcal R}\) is transitive and \(\ker T\) is invariant under \({\mathcal R}.\)
If this is not so, then for every \(n \geq 1\) there is a vector \(x_n \in B\) such that \(\|R\| \geq n \), whenever \(R \in {\mathcal R}\) and \(RTx_n \in B\). Since \({\mathcal R}\) is localizing, there is a subsequence \((x_{n_j})\) and  a sequence \((R_j)\) in \({\mathcal R}\) such that \(\|R_j\| \leq 1\) and \((R_jx_{n_j})\) converges in norm to some nonzero vector \(x \in X\).  We have \(TR_j=R_jT\) for all \(j \geq 1\), so that \((R_jTx_{n_j})\) converges to \(Tx\) in norm. Now \(Tx \neq 0\) because \(T\) is injective and \(x \neq 0\). Since \({\mathcal R}\) is transitive, there is an operator \(R \in {\mathcal R}\) such that \(RTx \in {\rm int}\,B\). It follows that there is a \(j_0 \geq 1\) such that \(RR_jTx_{n_j} \in {\rm int}\, B\) for every \(j \geq j_0\). Since \(RR_j \in {\mathcal R}\), the choice of the sequence \((x_n)\) implies that \(\|RR_j\| \geq n_j\) for every \(j \geq j_0\), and this is a contradiction because \(\|RR_j\| \leq \|R\|\) for every \(j \geq 1\).
\end{proof}
 If \(E\) is a Banach space then \(E^\ast\) denotes its dual space.
If \({\mathcal R} \subseteq {\mathcal B}(E)\) is a subalgebra, then  \({\mathcal R}^\ast \) denotes the subalgebra of \({\mathcal B}(E^\ast)\) of the adjoints of the elements of \({\mathcal R},\)  that is \({\mathcal R}^\ast = \{R^\ast: R \in {\mathcal R}\}.\) 
\begin{theorem}\cite[Theorem 2.3]{LRT}
\label{multi}
Let \(E\) be a complex Banach space, let \({\mathcal R} \subseteq {\mathcal B}(E)\) be a strongly compact  subalgebra such   that \({\mathcal R}^\ast\) is a transitive localizing algebra and it is closed in the weak-\(\ast\) operator topology. If \(T \in {\mathcal R}^\prime\) is a non zero operator then there is an operator \(R \in {\mathcal R}\) and there is a non zero vector \(x^\ast \in E^\ast\) such that \(R^\ast T^\ast x^\ast=x^\ast.\) Moreover,  the operator  \(T^\ast\) has a non trivial invariant subspace.
\end{theorem}
\begin{proof}
We shall apply Proposition \ref{translocal} to the algebra \({\mathcal R}^\ast.\) Let \(B^\ast \subseteq E^\ast\) be a closed  ball as in the definition of a localizing algebra,
let \(r>0\)  be a positive number as in Proposition \ref{translocal},  and define a multivalued map \(\Phi:B^\ast \to {\mathcal P}(B^\ast)\) by the expression 
\[
\Phi(x^\ast)=r  \,{\rm ball}({\mathcal R}^\ast)(T^\ast x^\ast) \cap B^\ast.
\]
Then,  \(\Phi(x^\ast)\) is a non empty, convex subset of \(B^\ast.\) Also, \(\Phi(x^\ast)\) is weak-\(\ast\) closed because \({\rm ball}({\mathcal R}^\ast)(T^\ast x^\ast) \) is weak-\(\ast\) compact as the image of  \({\rm ball}({\mathcal R}^\ast)\) under the map \(R^\ast \to R^\ast T^\ast  x^\ast,\) which is continuous from \({\mathcal B}(E^\ast)\)  with the weak-\(\ast\) operator topology into \(E^\ast\)  with  the weak-\(\ast\)  topology, and  \({\rm ball}({\mathcal R}^\ast)\) is compact in the weak-\(\ast\) operator topology. 

We claim that \(\Phi\) is upper semicontinuous for the weak-\(\ast\) topology. Indeed, let \(x^\ast,y^\ast \in B^\ast,\) and let \((x^\ast_\alpha)\) and \((y^\ast_\alpha)\) be two nets in \(B^\ast\)  with \(x^\ast_\alpha \to x^\ast,\)  \(y^\ast_\alpha \to y^\ast\) in the weak-\(\ast\) topology, and such that  \(y^\ast_\alpha \in \Phi(x^\ast_\alpha).\) We must show that \(y^\ast \in \Phi(x^\ast).\) Since \(y^\ast_\alpha \in \Phi(x^\ast_\alpha),\) there is an \(R^\ast_\alpha \in  \,{\rm ball}({\mathcal R}^\ast)\) such that \(y^\ast_\alpha=r R^\ast_\alpha T^\ast  x^\ast_\alpha.\) Since \({\rm ball}({\mathcal R})\) is precompact in the strong operator topology, there exists a subnet \((R_{\alpha_\beta})\) that converges in the strong operator topology to some \(R \in {\mathcal B}(E).\) Thus, \(R^\ast_{\alpha_\beta} \to R^\ast\) in the weak-\(\ast\) operator topology. Notice that \({\rm ball}({\mathcal R}^\ast)\) is compact in this topology because  \({\rm ball}({\mathcal B}(E^\ast))\) is compact in this topology and \({\mathcal R}^\ast\) is closed in this topology. It follows that \(R^\ast \in {\rm ball}({\mathcal R}^\ast).\) Let \(x \in E\) and notice that \(\|TR_{\alpha_\beta}x-TRx\| \to 0.\) Then
\begin{eqnarray*}
\langle x, y^\ast_{\alpha_\beta} \rangle & = & \langle x, r R^\ast_{\alpha_\beta}  T^\ast x^\ast_{\alpha_\beta} \rangle\\
& = & r \langle TR_{\alpha_\beta} x,  x^\ast_{\alpha_\beta} \rangle \\ 
& = & r \langle TR_{\alpha_\beta} x-TRx,  x^\ast_{\alpha_\beta} \rangle + r \langle TR x,  x^\ast_{\alpha_\beta} \rangle.
\end{eqnarray*}
We have  \(\langle TR_{\alpha_\beta} x-TRx,  x^\ast_{\alpha_\beta} \rangle \to 0\) and \(\langle TR x,  x^\ast_{\alpha_\beta} \rangle \to \langle TRx,x^\ast \rangle = \langle x,R^\ast T^\ast x^\ast \rangle,\) so that 
\[
\langle x, y^\ast_{\alpha_\beta} \rangle \to \langle x, r R^\ast T^\ast x^\ast \rangle.
\] 
Since \(x \in E\) is arbitrary, \(y^\ast_{\alpha_\beta} \to r R^\ast T^\ast x^\ast\) in the weak-\(\ast\) topology, and it follows that \(y^\ast = r R^\ast T^\ast x^\ast.\) This shows that \(y^\ast \in \Phi(x^\ast),\) and the proof of our claim is complete. 

Since the map \(\Phi\) is upper semicontinuous and \(B^\ast\) is compact in the weak-\(\ast\) topology, it follows from the Ky Fan fixed point theorem that there is a vector \(x^\ast \in B^\ast\) such that \(x^\ast \in \Phi(x^\ast),\) that is, there is an operator \(R \in {\rm ball}({\mathcal  R})\) such that \(x^\ast=r R^\ast T^\ast x^\ast.\) 
\end{proof}
\begin{corollary}\cite[Corollary 2.4]{LRT}
\label{sclocal}
Let \(T \in {\mathcal B}(E)\) be an operator such that \(\{T\}^\prime\) is a strongly compact algebra and \(\{T^\ast\}^\prime\) is a localizing algebra. Then \(T^\ast\) has a nontrivial invariant subspace.
\end{corollary}
\begin{proof}
If \(T^\ast\) has a hyperinvariant subspace then there is nothing to prove, and otherwise \(\{T^\ast\}^\prime\) is a transitive algebra, so that Theorem \ref{multi} applies.
\end{proof}
\noindent
Notice that the assumptions of Corollary \ref{sclocal} are met whenever \(T\) is a compact operator with dense range. 

\section{An infinite dimensional version of Burnside's theorem}
\label{burnside}
\noindent
Burnside's classical theorem is the assertion that for  a finite dimensional linear space  \(F,\)  the only transitive subalgebra of  \( {\mathcal B}(F)\) is the whole algebra \( {\mathcal B}(F).\)  Lomonosov \cite{lomonosov91} obtained a generalization of Burnside's theorem to infinite dimensional Banach spaces. Scott Brown \cite{brown} proved the same result independently for  the special case of a Hilbert space and a commutative algebra. Lindstr\"om and Schl\"uchtermann \cite{LS} provided a  relatively short proof of the Lomonosov result in full generality. In this section we present a proof of the Scott Brown result that is based on Schauder fixed point theorem. 

Let \(H\) be a complex, infinite dimensional, separable Hilbert  space. Let \(T \in {\mathcal B}(H)\) and let \(\|T\|_e\) denote the essential norm of \(T,\) that is, the distance from \(T\) to the space of compact operators.

\begin{theorem} \cite[Theorem 1.1]{brown}
\label{burnthm}
Let \({\mathcal R}\)  be a commutative subalgebra of \({\mathcal B}(H)\). Then there exist  nonzero vectors \(x, y \in  H\) such that for any \(R \in {\mathcal R}\) we have \( |\langle Rx , y \rangle | \leq \|R\|_e.\)
\end{theorem}
\begin{proof}
Consider the set \({\mathcal E} = \{R \in {\mathcal R}: \|R\|_e \leq 1/16\}.\)  We claim that there is some \(x \in H \backslash \{0\}\) such that the set  \({\mathcal E}x\) is not dense in \(H.\) The result then follows easily  because in that case there is some \(y \in H \backslash \{0\}\) such that   \( |\langle Rx , y \rangle | \leq 1\) for all \(R \in {\mathcal E}.\)
Now, for the proof of our claim, we proceed by contradiction. Suppose that  the set  \({\mathcal E}x\) is  dense in \(H\) for every  \(x \in H \backslash \{0\}.\)  Choose \(x_0 \in H\) with \(\|x_0\|=2\) and consider the closed ball \(B=\{x \in H: \|x-x_0\| \leq 1\}.\) Then, for every  vector \(x \in B\) there is an operator  \(R_x \in {\mathcal E}\) such that \(\|R_xx -x_0\| < 1/2.\) Next, there is a bounded operator \(T_x\) and a compact operator \(K_x\) such that \(R_x=T_x+K_x\) and \(\|T_x\| \leq 1/8.\) Since \(K_x\) is a compact operator, it is weak to norm continuous on bounded sets, so that there exists an open neighborhood of \(x\) in the weak topology, say  \(V_x \subseteq H,\) such that \(\|K_xy-K_xx\| < 1/4\) for all \(y \in V_x\ \cap B.\) Then consider the set \(U_x=V_x \cap B\) and notice that \(U_x\) is an open neighborhood of \(x\) in the weak topology relative to \(B.\) Moreover, for \(y \in U_x\) we have
\[
\|R_xy-R_xx\| \leq \|T_xy-T_xx\| + \|K_xy-K_xx\| < 2 \cdot \frac{1}{8} + \frac{1}{4} = \frac{1}{2},
\]
and therefore \(\|R_xy-x_0\| < 1.\) Hence, \(R_x U_x \subseteq B.\)  Since \(B\) is compact in the weak topology, there exist finitely many vectors \(x_1, \ldots x_n \in B\) such that
\[
B \subseteq \bigcup_{j=1}^n U_{x_j}.
\]
Choose some weakly continuous functions \(f_1, \ldots ,f_n\) on \(B\) such that \({\rm supp}( f_j) \subseteq U_j,\) \(0 \leq f_j(x) \leq 1,\)  and
\[
\sum_{j=1}^n f_j(x)=1 \qquad \text{for all } x \in B.
\]
Define a weakly continuous mapping  \(\Phi:B \to B\) by the expression
\[
\Phi(x)= \sum_{j=1}^n f_j(x)R_{x_j}x \qquad \text{for all } x \in B,
\]
and apply the Schauder fixed point theorem to find a vector \(y_0 \in B\) such that \(\Phi(y_0)=y_0.\) Finally, consider the operator \(R \in {\mathcal R}\) defined by the expression
\[
R= \sum_{j=1}^n f_j(y_0)R_{x_j}.
\]
Hence, \(Ry_0=y_0.\) Notice that \(R \neq I\) because \(\|R\|_e \leq 1/8.\) Then, the eigenspace \(M= \{x \in H: Rx=x\}\) is a closed nontrivial invariant subspace for the algebra \({\mathcal R}.\)  Thus, any vector \(x \in M\) has the property that the set \({\mathcal E}x\) is not dense in \(H.\) The contradiction has arrived.
\end{proof}
\section{Invariant subspaces for operators on Krein space}
\label{krein}
Let \(H_1,H_2\) be two Hilbert spaces and consider the orthogonal direct sum \(H=H_1 \oplus H_2.\) Let \(P_1,P_2\) denote the orthogonal projections from \(H\) onto \(H_1,H_2,\) respectively. Consider  the operator \(J:=P_1-P_2.\) The {\em Krein space} is the space \(H\) provided with the indefinite product
\[
[x,y]:= \langle Jx, y \rangle, \qquad x,y \in H.
\]
Notice that  \(J\) is a selfadjoint involution, that is, \(J^\ast = J\) and \(J^2=I.\) The operator \(J\) is sometimes called the {\em fundamental symmetry} of the Krein space.

A vector \(x \in H\) is said to be {\em non negative} provided that \([x,x] \geq 0,\) and a subspace \(M \subseteq H\) is said to be {\em non negative} provided that  \([x,x] \geq 0\) for all \(x \in M.\)

Every operator \(T \in{\mathcal B}(H)\) has a matrix representation
\[
T= \left [ \begin{array}{cc}
		T_{11} & T_{12}\\
		T_{21} & T_{22}
		\end{array}
		\right ]
\]
with respect to the decomposition  \(H=H_1 \oplus H_2.\)

There is a natural, one to one and onto correspondence between the maximal non negative invariant subspaces \(M\) of an operator  \(T \in{\mathcal B}(H)\) and the contractive solutions \(X \in  {\mathcal B}(H_1, H_2)\) of the so called {\em operator Riccati equation}
\[
XT_{12}X + XT_{11}- T_{22}X -T_{21}=0.
\]
The correspondence \( X \leftrightarrow M\) is given by \(M= \{ x_1 \oplus X x_1: x_1 \in H_1\},\) where \(\|X\| \leq 1.\) The operator \(T\) is usually called the {\em hamiltonian operator} of the operator Ricatti  equation. 

An operator \(T \in{\mathcal B}(H)\) is said to be \(J\)-{\em selfadjoint} provided that \([Tx,y]=[x,Ty]\) for every \(x,y \in H.\) This is equivalent to say that \(JT=T^\ast J,\) or in other words, \(T_{11}^\ast =T_{11},\) \(T_{22}^\ast =T_{22},\) and \(T_{12}^\ast =-T_{21}.\)

A classical theorem of Krein is the assertion that, if the hamiltonian operator \(T\) is \(J\)-selfadjoint and the corner operator \(T_{12}\) is compact, then there exists a maximal non negative invariant subspace for \(T.\)

Albeverio, Makarov, and Motovilov \cite{AMM} addressed the question of the existence and uniqueness of contractive solutions to the operator Riccati equation under the condition that the diagonal entries in the hamiltonian operator have disjoint spectra, that is, \(\sigma(T_{11}) \cap \sigma(T_{22})= \emptyset.\) They proved the following

\begin{theorem}\cite[Theorem 3.6 and Lemma 3.11]{AMM}
\label{universal}
There is some universal constant \(c > 0\) such that whenever the corner operator \(T_{12}\) satisfies the condition
\[
\|T_{12}\| < c \cdot {\rm dist} [\sigma(T_{11}), \sigma(T_{22})],
\]
there is a unique solution \(X\) to the operator Riccati equation with \(\|X\| \leq 1.\)
\end{theorem}
\noindent
An earlier result in this direction was given by Motovilov \cite[Corollary 1]{motovilov} with the stronger assumption that the corner operator \(T_{12}\) is Hilbert-Schmidt. Adamjan, Langer, and Tretter \cite{ALT} extended  the technique to the case that the hamiltonian operator is not \(J\)-selfadjoint. Kostrykin, Makarov, and Motovilov \cite{KMM} adopted the assumption that \(\sigma(T_{11})\)  lies in a gap of \(\sigma(T_{22})\) and they showed that the best constant, in that context,  is \(c = \sqrt{2}.\)

We present a proof of Theorem \ref{riccati} that is based on Banach fixed point theorem. This method  can be found in the paper of Albeverio, Motovilov and Shkalikov \cite[Theorem 4.1]{AMS}. A basic tool is the bounded linear operator \(R\) defined for \(X \in {\mathcal B}(H_1,H_2)\) by the expression
\[
R(X):=T_{22}X-XT_{11}.
\]
It follows from the Rosenblum theorem that the map \(R\) is invertible. The main result is the following
\begin{theorem}\cite[Theorem 4.1]{AMS}
\label{riccati}
If the operators \(T_{11}, T_{22}\) have disjoint spectra and the corner operator \(T_{12}\) satisfies the estimate
\[
\|T_{12}\| < \frac{1}{2 \|R^{-1}\|}
\]
then there is a unique solution \(X\) to the operator Riccati equation with \(\|X\| \leq 1.\)
\end{theorem}
\noindent
The following upper bound on the norm of the inverse \(R^{-1}\) can be found in the work of Albeverio, Makarov, and Motovilov \cite[Theorem 2.7]{AMM}. See also the paper by Bhatia and Rosenthal [4, p.15] for this interesting result and other related issues.
\begin{theorem}\cite[Theorem 2.7] {AMM}
\label{inverse} 
If the operators \(T_{11},T_{22}\) have disjoint spectra then
\[
\| R^{-1}\| \leq \frac{\pi}{2} \cdot \frac{1}{{\rm dist}[\sigma(T_{11}),\sigma(T_{22})]}
\]
\end{theorem}
\noindent
Notice that Theorem \ref{universal} becomes a corollary of Theorem \ref{riccati} and Theorem \ref{inverse} with the  constant \(c=1/\pi.\)
\begin{proof}[Proof of Theorem \ref{riccati}.] Consider the quadratic map \(Q\) defined for \(X \in {\mathcal B}(H_1,H_2)\) by the expression
\[
Q(X):= XT_{12}X-T_{21}.
\]
It is clear that the operator Riccati equation can be expressed as
\[
Q(X)-R(X)=0,
\]
or equivalently, \(X=R^{-1}(Q(X)).\) Thus, the solutions of the operator Riccati equation are  the fixed points of the map \(S:=R^{-1} \circ Q.\) Now, let us check that the map \(S\) takes the unit ball of \({\mathcal B}(H_1,H_2)\) into itself. Indeed, if \(\|X\| \leq 1\) then
\begin{eqnarray*}
\|S(X)\| & = & \|R^{-1}(Q(X))\| \leq \|R^{-1}\| \cdot \|Q(X)\|\\
& \leq &  \|R^{-1}\| \cdot ( \|T_{12}\| \cdot \|X\|^2 + \|T_{21}\|)\\
& \leq & \|R^{-1}\| \cdot ( \|T_{12}\| + \|T_{21}\|) = 2 \|R^{-1}\| \cdot \|T_{12}\| <1.
\end{eqnarray*}
Also, the map \(S\) is contractive, for if \(\|X\|,\|Y\| \leq 1\) then
\begin{eqnarray*}
\|Q(X) - Q(Y)\| & = & \| XT_{12}X -Y T_{12}Y\|\\
& \leq & \| XT_{12}X -X T_{12}Y\| + \| XT_{12}Y -Y T_{12}Y\|\\
& \leq & ( \| X\| + \|Y \|) \cdot \|T_{12}\| \cdot \|X-Y\| \leq 2 \|T_{12}\| \cdot \|X-Y\|,
\end{eqnarray*}
and from this inequality it follows that
\begin{eqnarray*}
\|S(X)-S(Y)\| & = & \|R^{-1}(Q(X)-Q(Y))\|\\
& \leq & \|R^{-1}\| \cdot \|Q(X)-Q(Y)\| \leq 2 \|R^{-1}\| \cdot \|T_{12}\| \cdot \|X-Y\|,
\end{eqnarray*}
so that the map \(S\) satisfies a Lipschitz condition with a Lipschitz constant \(2 \|R^{-1}\| \cdot \|T_{12}\| < 1.\) The result now follows at once as a consequence of the  Banach fixed point theorem. 
\end{proof}

\bigskip

\section*{Competing interests}
The authors declare that they have no competing interests.

\section*{Author's contributions}
Both authors contributed equally and significantly in writing this article. Both authors read and approved the final manuscript.

\section*{Acknowledgements}
  \ifthenelse{\boolean{publ}}{\small}{}
This research was partially supported by Junta de Andaluc{\'\i}a under projects FQM-127 and FQM-3737, and by Ministerio de Educaci\'on, Cultura y Deporte under projects MTM2012-34847C02-01 and  MTM2009-08934. 

\newpage
{\ifthenelse{\boolean{publ}}{\footnotesize}{\small}
 \bibliographystyle{bmc_article}  
 \bibliography{espinola_lacruz} }     

\end{bmcformat}
\end{document}